\documentclass[10pt]{amsart}
\usepackage{fullpage}   

\usepackage{amsmath}
\usepackage{amsfonts}
\usepackage{amssymb}
\usepackage{amsthm}
\usepackage{color}
\usepackage{hyperref}

\newtheorem{theorem}{Theorem}
\newtheorem{proposition}{Proposition}
\newtheorem{lemma}{Lemma}
\newtheorem{corollary}{Corollary}
\newtheorem{observation}{Observation}

\theoremstyle{definition}
\newtheorem{definition}{Definition}

\theoremstyle{remark}

\newcommand{\E}{\mathcal{E}}

\newcommand{\R}{\mathcal{R}}
\newcommand{\C}{\mathcal{C}}
\newcommand{\A}{\mathcal{A}}

\newcommand{\BR}{\mathbb{R}}

\newcommand{\BZ}{\mathbb{Z}}
\newcommand{\BF}{\mathbb{F}}

\newcommand{\si}{\sigma}

\def\Per{\mathrm{per}}
\def\sign{\mathrm{sign}}
\def\rot{\mathrm{rot}}

\begin{document}



\title[]{Binary linear codes via 4D discrete Ihara-Selberg function}

\author{Martin Loebl}
\address{Department of Applied Mathematics, Charles University, Malostransk\'{e} nam. 25, 118 00 Praha 1, Czech Republic.}
\email{loebl@kam.mff.cuni.cz}

\thanks{Supported by the Czech Science Foundation under the contract number P202-13-21988S}

\date{\today}

\begin{abstract}  
We express the weight enumerator of each binary linear code, in particular the Ising partition function of an arbitrary finite graph, as a formal infinite product.
An analogous result was obtained by Feynman and Sherman in the beginning of the 1960's for the special case of the Ising partition function of the planar graphs. 
A product expression is an important step towards understanding the logarithm of the Ising partition function, for general graphs and in particular for the cubic 3D lattices.
\end{abstract}

\maketitle


\section{Introduction}
\label{sec.intro}

A linear code $\C$ of length $n$ over the binary field $\BF_2$ is a linear subspace of $\BF_2^n$. Each vector in $C$ is called a {\em codeword}. Let $w:\{1,\ldots, n\}\rightarrow \BR$ be a real weight function. The weight of a codeword $c$, denoted by $w(c)$, is defined as the sum of the weights of the non-zero entries of $c$. The weight enumerator of $(\C, w)$ is defined as 
$$
W_{\C, w}(x)= \sum_{c\in \C}x^{w(c)}.
$$
A k-uniform hypergraph (k-hypergraph for short) is a pair $H=(V,E)$ where $E$ is a set of k-element subsets of $V$ called {\em hyperedges}.
If a set is given along with its linear order, we say that the set is {\em directed}. A {\em directed} k-hypergraph is a pair $D=(V,A)$ where $A$ is a set of directed k-element-subsets of $V$. We note that one k-element-subset may appear several times in $A$, with different linear orders.
As an illustration we observe that directed 2-hypergraphs are exactly directed graphs. 
In this paper we show that 
\medskip\noindent
\newline
{\em The weight enumerator of each binary linear code can be expressed as a formal product in the form of the 4-dimensional discrete Ihara-Selberg function of a weighted directed 4-hypergraph (this novel concept is defined in Definition \ref{def.4disf}).}
\medskip\noindent
\newline
 The main theorem is as follows (see subsection \ref{sub.pm} for the proof).

\begin{theorem}
\label{thm.main}
Let $\C$ be a binary linear code of length $n$ over the binary field $\BF_2$, let $w:\{1,\ldots, n\}\rightarrow \BR$ be a real weight function and let $W_{\C, w}(x)$ be the weight enumerator of $(\C, w)$.
Then one can construct in polynomial time a directed 4-hypergraph $D= (V,A)$ and weight function $w': A\rightarrow \BR$, so that $W_{\C, w}(x)$ is equal, up to the sign, to the 4-dimensional discrete Ihara-Selberg function of $(D, w')$.
\end{theorem}

In next two subsections we present a motivating example.

\subsection {Motivating Example: Ising partition function of a planar graph}
\label{sub.is}
Let $G= (V,E)$ be a graph; we assume a graph is {\em simple} in this paper, i.e., it has no multiple edges or loops. We say that $E'\subset E$ is
{\em even} if the graph $(V,E')$ has even degree (possibly
zero) at each vertex.  Let $\mathcal{E}(G)$
denote the set of all even edge sets of $G$. The set of the characteristic vectors of the elements of 
$\mathcal{E}(G)$ is equal to the kernel (over $\BF_2$) of the incidence matrix of graph $G$ and thus it forms 
a binary linear code called {\em cycle space} of $G$ (see \cite{L}).

We assume that a variable $x_e$ is associated with each edge
$e$, and define the generating function for even sets,  $\E_G$, 
in $\BZ[(x_e)_{e\in E(G)}]$,  as follows: 
$$\E_G(x)= \sum_{E'\in\mathcal{E}(G) }\,\,\prod_{e\in E'}x_e~.$$

Given edge-weights $w: E\rightarrow \BR$, we obtain the {\em weight enumerator of the cycle space} of $G$ from 
$\E_G(x)$ by substituting $x_e:= y^{w(e)}$ for each edge $e\in E$.

 Knowing the polynomial $\E_G$ is equivalent to knowing the {\em Ising partition
 function} 
 $Z^{\mathrm{Ising}}_G$
 of the graph $G$ \cite{vdW} (see also \cite{L}).

In a suggestion how to prove the formula of  Kac and Ward \cite{KW} for the planar Ising partition function, Feynman (see \cite{S}) proposed that $\E_G(x)$ is equal to a {\em formal product}, which then equals to the square root of the determinant introduced in \cite{KW}. The formula of Feynman was proved by Sherman in \cite{S};
Sherman also studies in \cite{S} the logarithm of the Feynman's product and reproduces results for the free energy of the planar Ising problem.

\subsection{Continuation of the Example: Feynman's formula}
\label{s.fey}
We start with some basic definitions.
Let $D= (V,A)$ be a directed graph. A sequence of arcs $a_1, a_2, \ldots, a_n$ is a {\em rooted closed walk} if for each $i< n$ the terminal vertex of $a_i$ is equal to the initial vertex of $a_{i+1}$, and the terminal vertex of $a_n$ is equal to the initial vertex of $a_1$. A {\em closed walk} is an equivalence class of rooted closed walks where we identify two rooted closed walks 
$W_1, W_2$ if $W_1= a_1, a_2, \ldots, a_n$ and $W_2= a_2, a_3, \ldots, a_n, a_1$.

A closed walk $W$ is {\em reduced} if two orientations of the same edge are never consecutive. A closed walk $W$ is a {\em power} of closed walk $T$ if there is integer $k> 1$ such that $W$ is obtained by concatenating $k$ copies of $T$. We say that a closed walk is
{\em aperiodic} if it is not a power of a smaller closed walk. We denote by $\Delta(D)$ the set of all {\em aperiodic closed walks} of $D$.

Let $\R(D)$ denote the set of {\em equivalence classes of aperiodic reduced closed walks} of $D$:  We identify  pairs of aperiodic reduced closed walks that are reverses of each other. 
If $W\in \R(D)$ then we denote by $A(W)$ the multi-set of the arcs of $W$. We note that $W$ may have more than one copy of any arc of $D$. We note that $\R(D)$ is infinite whenever $D$ has two directed cycles which share a vertex.

\medskip\noindent

Let $G= (V,E)$ be a planar graph properly embedded in the plane and as above we assume that a variable $x_e$ is associated with each edge $e$. Let $D= (V,A)$ be the directed graph obtained 
by replacing each edge $\{u,v\}$ by the two arcs $(u,v)$ and $(v,u)$. If $a= (u,v)$ and $e=\{u,v\}$ then we let $x_a= x_e$.

We denote by $\rot(p)$ the {\em rotation}, sometimes called the {\em Whitney index} of reduced aperiodic closed walk $p$.  We use Feynman's formula only as an example of an expression we aim at, and so we do not include the technical definition of a rotation here;  see \cite{S} or book \cite{L} for a definition and basic facts. One basic fact  is that $(-1)^{\rot(p)}$ is well defined for the elements of $\R(D)$.  Feynman suggested and Sherman \cite{S} proved the following (see \cite{L} for a proof).

\begin{theorem}
\label{thm.fs}
Let $G= (V,E)$ be a planar graph properly embedded in the plane. Then
$$
\E_G(x)= \prod_{W\in \R(D)}[1- (-1)^{\rot(W)}\prod_{a\in A(W)}x_a].
$$
\end{theorem}

The product in Theorem \ref{thm.fs} is a {\em formal product} which we define as follows. If $I$ is a countable set and $f_i, i\in I$ are polynomials, then we let
$$
\prod_{i\in I}[1+f_i]= \sum_{J\subset I\text{ finite }}\prod_{j\in J}f_j.
$$

An example of a formal product directly relevant to the constructions of this paper is Bass' theorem.

\subsection{Bass' theorem}
\label{sub.bass}
If $D= (V,A)$ is a directed graph with no loops then let $\S(D)$ denote the set of aperiodic closed walks of $D$. We assume that a variable $x_a$ is associated with each arc $a$ of $D$ and let $x= (x_a)_{a\in A}$. The {\em adjacency matrix} $A(D,x)$ is the $V\times V$ matrix defined by $A(D,x)_{uv}= x_{(uv)}$ when $(u,v)\in A$, and $A(D,x)_{uv}= 0$ otherwise.

Bass' theorem (see \cite{B} and \cite{FZ} for several proofs and generalizations) expresses the determinant of $I-A(D)$, where $I$ is the identity matrix and 
$A(D)$ is the adjacency matrix of a digraph $D$, as a formal product over aperiodic closed walks in $D$.

\begin{definition}
\label{def.isf}
Let $D= (V,A)$ be a directed graph with no loops and let $x= (x_a)_{a\in A}$ be the vector of variables. The {\em discrete Ihara-Selberg function} $IS_D(x)$ is the following formal product:
$$
IS_D(x)= 
\prod_{W\in \Delta(D)} (1- \prod_{a\in A(W)}x_a).
$$ 
\end{definition}

\begin{theorem}[Bass' theorem]
\label{thm.ba}
Let $D= (V,A)$ be a directed graph with no loops and let $x= (x_a)_{a\in A}$ be the vector of variables. Then
$$
\det(I-A(D,x))= IS_D(x).
$$
\end{theorem}

We recall that set $\Delta(D)$ is infinite whenever $D$ has two directed cycles which share a vertex.
We denote by $F$ the set of all {\em finite subsets} of $\Delta(D)$. By the definition of a formal product we have 
\begin{equation}
\label{eq.p}
IS_D(x)= \sum_{\{W_1, \ldots, W_k\}\in F} (-1)^{k} \prod_{i=1}^k\prod_{a\in A(W_i)}x_a
\end{equation}

Let $F'$ denote the subset of $F$ 
consisting of all sets of vertex disjoint directed cycles of $D$. The following is a well-known fact:
\begin{equation}
\label{eq.d}
\det(I- A(D,x))= \sum_{\{W_1, \ldots, W_k\}\in F'} (-1)^{k} \prod_{i=1}^k\prod_{a\in A(W_i)}x_a
\end{equation}

Hence one can view Bass' theorem as a strengthening of (\ref{eq.d}).

\subsection{Main contribution}
\label{sub.mma}
For a non-planar graph $G$, the Ising partition function can be written as a linear combination of formal products \cite{L1} (see also \cite{L}), but a single product formulation, crucial for studying the logarithm and the free energy,  has been missing.  The initial idea of this paper is that for the Ising model where the underlying graph is non-planar, we should 
consider underlying structure of higher dimension than that of graphs. 
It turns out that it is advantageous to study the weight enumerators of all binary linear codes, i.e., in particular the Ising partition functions of general graphs, as functions on uniform hypergraphs. The relevant functions are hyperdeterminants and hyperpermanents of
higher-dimensional matrices. Finally, a product formula is obtained by the 4-dimensional analogue of Bass' theorem.

\medskip\noindent

The rest of the paper is
devoted to the proof of Theorem \ref{thm.main}.
In section \ref{sec.geom} we show
that the weight enumerator of each binary linear code can be written as $\det(I+A)$, where $A$ is the incidence matrix of a directed 4-uniform hypergraph, and $I$ is the 4-dimensional identity matrix. This section generalizes the theory of Kasteleyn orientations and uses results of \cite{LR, R1, R2}.  Then, in section \ref{s.3bt}, the 4-dimensional discrete Ihara-Selberg function is defined and the 4-dimensional analogue of Bass' theorem is stated and it is proved in the last section. 
The results of the last three sections prove the main theorem.

\medskip\noindent

{\bf Acknowledgement.} This Project started with a visit of Iain Moffatt in Praha in 2013. I would 
like to thank to him for the enthusiastic discussions during the visit. The idea to formalize 
circulations by defining connectors belongs to Iain. I would also like to thank to Martin Klazar for discussions on formal products.

\section{Geometric representation of binary linear codes}
\label{sec.geom}

 We say that a k-hypergraph $H= (V,E)$ is {\em almost disjoint} if each pair of hyperedges intersects in at most one vertex. We say that $H$ is {\em k-partite}
if the vertex-set $V$ is partitioned into $k$ subsets $V_1, \ldots, V_k$ and each hyperedge of $E$ intersects each $V_i$ in exactly one vertex.  $M\subset E$ is a {\em perfect matching} of $H$ if the elements of $M$ are pairwise disjoint and $\cup M= V$.

\subsection{Hypermatrices, their determinants and permanents}
\label{sub.hyp}
We start by introducing basic notions.
\begin{itemize}
\item
A {\em k-matrix} is an array indexed by all k-tuples from its {\em index set} $V_1\times \ldots\times V_k$. 
\item
Let 
$D= (V,A)$ be a directed k-hypergraph and  let $x= (x_a)_{a\in A}$ be the vector of variables associated with directed hyperedges of $D$. The {\em adjacency matrix} $A(D, x)$ is the k-matrix with index-set $V^k$ defined by: $A(D,x)_{a}= x_a$ for each $a\in A$, and $A(D,x)_{a}= 0$ otherwise. If a weight function $w: A\rightarrow \BR$ is given and z is a variable then we let 
$A(D,w,z)= A(D,x)|_{x_a:= z^{w(a)}, a\in A}$.
\item
Let $H=(V_1, \ldots, V_k, E)$ be a k-partite k-hypergraph and  let $x= (x_e)_{e\in E}$ be the vector of variables associated with hyperedges of $H$. The {\em transition matrix} $T(H, x)$ is the k-matrix with index-set $U= V_1\times\ldots \times V_k$ defined by: 
$T(H,x)_{u}= x_u$ for each $u\in E$, and $T(H,x)_{u}= 0$ otherwise.  If a weight function $w: E\rightarrow \BR$ is given and z is a variable then we let 
$T(H,w,z)= T(H,x)|_{x_e:= z^{w(e)}, e\in E}$.
\end{itemize}

If $M$ is a (k+1)-matrix with index-set $U= V_1\times\ldots \times V_k\times V_{k+1}$ then its {\em determinant} is defined by
$$
\det(M)= \sum_{\alpha_1, \ldots, \alpha_k} \prod_{1\leq j\leq k}sign(\alpha_j) 
\prod_{i\in V_0}M_{(\alpha_1(i),\ldots, \alpha_k(i), i)},
$$
where in the summation each $\alpha_i$ is a permutation of $V_i$.
Analogously, the {\em permanent} is defined by 
$$
\Per(M)= \sum_{\alpha_1, \ldots, \alpha_k}  
\prod_{i\in V_0}M_{(\alpha_1(i),\ldots, \alpha_k(i), i)}.
$$

As an illustration we observe: if $H= (V_1, \ldots, V_k, E)$ is a k-partite k-hypergraph and $x= (x_e)_{e\in E}$ is the vector of variables associated with hyperedges of $H$ then $\Per(T(H,x))$ is the generating function of perfect matchings of $H$, i.e.,
$$
\Per(T(H,x))= \sum_{M\text{ perfect matching}} \prod_{e\in M}x_e.
$$

Finally, we introduce a useful construction how to get a directed k-hypergraph from an almost disjoint k-partite k-hypergraph. This construction generalises a well-known construction of a directed graph from a bipartite graph without parallel edges and with a fixed perfect matching. 

\begin{definition}
\label{def.contr}
Let $H= (V_1, \ldots, V_k, E)$ be an almost disjoint k-partite k-hypergraph and let $P$ be a perfect matching of $H$.  For $e\in E$ and $1\leq i\leq k$ let $e_i$ denote the vertex of $e\cap V_i$, let $P(e)_i$ denote the edge of $P$ which contains vertex $e_i$ and let $a(e,P)_i$ be the vertex of $P(e)_i\cap V_1$.
We let $a(e,P)$ be the directed hyperedge $(a(e,P)_1, \ldots, a(e,P)_k)$. We define directed k-hypergraph 
$D(H,P)= (V_1,A)$ so that $A= \{a(e,P): e\in E\setminus P\}$.
\end{definition}

We note that each $a(e,P)$ in the above definition is an ordered k-tuple of different vertices of $V_1$ since hypergraph $H$ is almost disjoint. Next observation follows directly from the definitions.
 
\begin{observation}
\label{o.contr}
Let $H= (V_1, \ldots, V_k, E)$ be an almost disjoint k-partite k-hypergraph and let $P$ be a perfect matching of $H$. Let $D(H,P)= (V_1, A)$ be the directed k-hypergraph of Definition \ref{def.contr}. Let $x= (x_e)_{e\in E}$ be the vector of variables associated with hyperedges of $H$ and let 
$y= (y_a)_{a\in A}$ be the vector of variables associated with directed hyperedges of $D(H,P)$. 
We let $x_p:= 1$ for each $p\in P$ and $y_{a(e,P)}:= x_e$ for each edge of $e\setminus P$. 
Further we assume set $V_1$ be linearly ordered. This induces linear order on $P$: first is the edge of P which contains the first vertex of $V_1$ and so on.  Linear order on $P$ induces linear order on each $V_i, i> 1$: first is the vertex of $V_i$ that belongs to the first edge of $P$ and so on. Then
$$
T(H,x)= I+A(D(H,P),y).
$$
\end{observation}

The main result of this section is the following theorem. Its proof is postponed to the end of the section.

\begin{theorem}
\label{thm.main1}
Let $\C$ be a binary linear code of length $n$ over the binary field $\BF_2$, let $w:\{1,\ldots, n\}\rightarrow \BR$ be a real weight function and let $W_{\C, w}(x)$ be the weight enumerator of $(\C, w)$.
Then one can construct in polynomial time a directed 4-hypergraph $D= (V,A)$ and weight function $t: A\rightarrow \BR$, so that  
$$
W_{\C, w}(z)= \pm\det(I+A').
$$
where $A'$ is a {\em signing} of the adjacency matrix $A(D,t,z)$, i.e., is obtained from $A(D,t,z)$ by multiplying some entries by $-1$.
\end{theorem}

\subsection{Continuation of the Example: Kasteleyn orientations}
\label{sub.disf}
If $\C$ is the cycle space of a planar graph then the hyperdeterminant of Theorem \ref{thm.main1} can be replaced by a determinant-type expression of a standard matrix, called a {\em Pfaffian}. This seminal construction of Kasteleyn \cite{K} and Fisher \cite{F} from the beginning of the 1960's started the theory of Kasteleyn orientations.  We briefly sketch the two steps of the construction below.
Define the generating function of perfect matchings of a graph $G$, $\P_G$,
as follows:
$$\P_G(x)= \sum_{E'\text{ perfect matching} }\,\,\prod_{e\in E'}x_e~.$$

Let $G= (V,E)$ be a planar graph with edge-variables $x= (x_e)_{e\in E}$.  

{\bf Step 1.} Fisher \cite{F} (see also \cite{L}) constructed a planar graph $H(G)= H$ so that a subset of its edges $E_0\subset E(H)$ is identified with the edge-set of $G$ and moreover: if we let $y_e, e\in E(H)$ be edge-variables of $H$ and let $y_e:= x_e$ for $e\in E_0$ while $y_e:= 1$ otherwise then 
$\E_G(x)= \P_H(y)$.

{\bf Step 2.} Given a planar graph $K$ with edge-variables $z= (z_e)_{e\in E(K)}$, Kasteleyn \cite{K} (see also \cite{L}) constructed an orientation $D_K$ of $K$ called {\em Kasteleyn} or {\em Pfaffian orientation} so that $\P_K(z)$ is equal 
to the Pfaffian of a matrix $M(D_K, z)$ defined by orientation $D_K$. 

Summarizing both steps, the weight enumerator $\E_G(x)$ of the cycle space of a planar graph is expressed as a Pfaffian.

\subsection{Weight enumerators as 3-permanents}
\label{sub.3per}
This subsection contains results of previous papers on geometric representations of binary linear codes \cite{R1} (master thesis of Rytir) and \cite{LR}.

Let $H= (V,E)$ be an almost disjoint k-hypergraph and let $t:E\rightarrow \BR$ be a weight function. We denote by $I(H)$ the $V\times E$ {\em incidence matrix} of $H$ defined as follows:
$I(H)_{v,e}= 1$ if $v\in e$ and $I(H)_{v,e}= 0$ otherwise. Next theorem asserts that the weight enumerator of each binary linear code has a geometric representation by a system of triangles.

\begin{theorem}
\label{thm.r1}
Let $\C$ be a binary linear code of length $n$ over the binary field $\BF_2$, let $w:\{1,\ldots, n\}\rightarrow \BR$ be a real weight function and let $W_{\C, w}(x)$ be the weight enumerator of $(\C, w)$.
Then one can construct in polynomial time an almost disjoint 3-hypergraph $H= (V,E)$ and a weight function $w': E\rightarrow \BR$ so that $W_{\C, w}(x)$ is equal to the weight enumerator
$W_{K, w'}(x)$, where $K= Ker_{\BF_2}(I(H))$ is the kernel (in $\BF_2$) of the incidence matrix $I(H)$.
\end{theorem}
\begin{proof}
This follows from Theorem 6 of \cite{R1} when we let (using notation of Theorem 6) $w'(B^n_i)= w(i), 1\leq i\leq n$, and $w'(e)= 1$ for each remaining triple $e\in E$.

\end{proof}

We remark that Theorem \ref{thm.r1} was generalized to linear codes over prime fields in the doctoral thesis of Rytir (see \cite{R2}). This suggests that results of this paper may be true for
all linear codes over prime fields.

Let $H=(V,E)$ be an almost disjoint k-hypergraph and let $t:E\rightarrow \BR$ be a weight function. We let

$$\P_{H,t}(x)= \sum_{E'\text{ perfect matching} }\,\,\prod_{e\in E'}x^{t(e)}~.$$

\begin{theorem} \cite{R1}
\label{thm.r2}
Let $H=(V,E)$ be an almost disjoint 3-hypergraph and let $t:E\rightarrow \BR$ be a weight function. Then one can construct an almost disjoint 3-hypergraph $H'= (V',E')$ and a weight function $t': E'\rightarrow \BR$ so that: if we denote $K= Ker_{\BF_2}(I(H))$ then 
$$
W_{K, t}(x)= \P_{H',t'}(x).
$$
Moreover, $H'$ has a perfect matching $P$ such that $t'(e)= 0$ for each $e\in P$.
\end{theorem}
\begin{proof}
The first part of the theorem is Theorem 20 of \cite{R1}. The existence of perfect matching $P$ follows easily from subsection 5.6 of \cite{R1} where a weight-preserving bijection between codewords of $K$ and perfect matchings of $H'$ is constructed. The desired perfect matching $P$ corresponds to the zero of the code $K$.

\end{proof}

Theorem 4 of \cite{LR}  strengthens Theorem \ref{thm.r2} as follows:

\begin{theorem}\cite{LR}
\label{thm.r3}
The statement of Theorem \ref{thm.r2} remains true if we require in addition that $H'$ is 3-partite.
 \end{theorem}
\begin{proof}
This is a direct consequence of  Theorem 4 of \cite{LR} and Theorem \ref{thm.r2}.

\end{proof}

The results of this subsection are summarized in the following:

\begin{theorem}
\label{thm.mainper}
Let $\C$ be a binary linear code of length $n$ over the binary field $\BF_2$, let $w:\{1,\ldots, n\}\rightarrow \BR$ be a real weight function and let $W_{\C, w}(x)$ be the weight enumerator of $(\C, w)$.
Then one can construct in polynomial time an almost disjoint 3-partite 3-hypergraph $H= (V,E)$, its perfect matching $P$ and weight function $w': A\rightarrow \BR$ so that $w'(p)= 0$ for each $p\in P$ and
$$
W_{\C, w}(x)= \Per(T(H,w')).
$$ 
\end{theorem}
\begin{proof}
This follows directly from Theorem \ref{thm.r1} and Theorem \ref{thm.r3}.

\end{proof}

\subsection{Kasteleyn hypermatrices, 3-permanents and 4-determinants}
\label{sub.4det}

\begin{definition}
\label{def.3ka}
We say that a k-matrix $A$ is {\em Kasteleyn} if
there is k-matrix $A'$ obtained from $A$ by changing signs of some entries, i.e., $A'$ is a signing of $A$, so that
$\Per (A)= \det(A')$.
\end{definition}

If $A$ is the transition matrix of a 2-partite 2-hypergraph, i.e., of a bipartite graph, then already Kasteleyn \cite{K} noticed that $A$ is Kasteleyn provided the bipartite graph is {\em planar}. 
Kasteleyn 2-matrices were characterized in a seminal paper \cite{RST}: the set of Kasteleyn 2-matrices is severely restricted and does not go far beyond the transition matrices of planar bipartite graphs.  
An important observation of \cite{LR} is that Kasteleyn 3-matrices form a rich class. 

 Theorem \ref{thm.k1} below generalizes Theorem 6 of \cite{LR} from Kasteleyn 3-matrices to Kasteleyn k-matrices, $k\geq 3$. The idea of the proof for general $k\geq 3$ is the same as for $k=3$ but we decided to include the proof in order to make the paper easier to follow.

 Let $A$ be a $V_1\times\ldots \times V_k\times V_{k+1}$ hypermatrix, $k\geq 2$, and $|V_i|= |V_j|$ for all $i\neq j$.  We first define k bipartite graphs $G_1, \ldots, G_k$ as follows.
We let, for $1\leq i\leq k$, $G_i= (V_{k+1}, V_i, E_i)$ where 
$$
E_i= \left\{\{v_{k+1}, v_i\}: v_{k+1}\in V_{k+1}, v_i\in V_i \text{ and there is } v= (v_1, \ldots, v_{i-1}, v_i,\ldots, v_{k+1}) \text{ such that } A_v\neq 0\right\}.
$$

\begin{theorem}
\label{thm.k1}
If $A$ is such that all $G_1, G_2,\ldots, G_k$ are planar bipartite graphs then $A$ is Kasteleyn.
\end{theorem}
\begin{proof}
Let $M_i$ be the transition matrix of $G_i$ and let $\sign_i:E(G_i)\mapsto\{-1,1\}$ be the signing of entries of $M_i$ which defines matrix $M_i'$ satisfying $\Per(M_i)= \det(M_i')$. We recall that such signing exists due to the above mentioned result of Kasteleyn since $G_i$ is a planar bipartite graph.
We define hypermatrix $A'$ by 
$$
A'_{(v_1, \ldots, v_k, v_{k+1})}= [\prod_{1\leq i\leq k}\sign_i(\{v_i, v_{k+1}\})]A_{(v_1, \ldots, v_{k+1})}.
$$
We have directly from the definition of the determinant and of $A'$ that
$$
\det(A')= \sum_{\si_1, \ldots, \si_{k-1}}(\prod_{1\leq i< k}\sign(\si_i))\times
\sum_{\si_k}\sign(\si_k)
\prod_{j\in V_{k+1}} \sign_k(\{\si_k(j),j\})\times
[\prod_{1\leq i< k}\sign_i(\{\si_i(j),j\})]A_{(\si_1(j)\ldots\si_k(j),j)}.
$$
By the construction of $\sign_k$ we have that for each $\si_k$ and each $\si_1, \ldots, \si_{k-1}$, if
$\prod_{j\in V_{k+1}} A_{(\si_1(j)\ldots, \si_k(j),j)}\neq 0$ then
$$
\sign(\si_k)\prod_{j\in V_{k+1}} \sign_k(\{\si_k(j),j\})= 1
$$
since $\{\{\si_k(j),j\}: j\in V_{k+1}\}$ is a perfect matching of $G_k$.
Hence
$$
\det(A')= \sum_{\si_k}\sum_{\si_1, \ldots, \si_{k-1}} (\prod_{1\leq i< k}\sign(\si_i))
\prod_{j\in V_{k+1}}[\prod_{1\leq i< k}\sign_i(\{\si_i(j),j\})]A_{(\si_1(j)\ldots\si_k(j),j)}.
$$

Continuing this way for $k-1, \ldots, 1$ we get
$\det(A')= \Per(A)$.

\end{proof}

The main result of this subsection is the following theorem. We remark that Theorem 5 of \cite{LR} shows a {\em weaker statement} for bipartite graphs, i.e., for 2-partite 2-hypergraphs; namely, the resulting 3-partite 3-hypergraph is not necessarily {\em almost disjoint} in Theorem 5 of \cite{LR}. Hence a new construction has to be invented to show:

\begin{theorem}
\label{thm.maindet}
Let $H=(V,E)$ be an almost disjoint 3-partite 3-hypergraph, let $P$ be a perfect matching and let $w: E\rightarrow \BR$ be a weight function so that $w(p)= 0$ for each $p\in P$. Then one can construct in polynomial time an almost disjoint 4-partite 4-hypergraph $H'= (V',E')$, a perfect matching $P'$ and a weight function $w': E'\rightarrow \BR$ so that $w'(p)= 0$ for each $p\in P'$ and
$$
\Per(T(H,w,z))= \det(T'),
$$ 
where $T'$ is a signing of $T(H',w',z)$.
\end{theorem}
\begin{proof}
Let $V= (V_1, V_2, V_3)$ be the 3-partition of $V$. We note that existence of a perfect matching implies that all three parts have the same size which we denote by $n$. We let $X= \{1,2,3,4\}$ and initially we let $V_i\subset V_i(H')$ for each $i= 1,2,3$. The construction consists in two steps.

{\bf Step 1.} For each $e= \{v_1,v_2,v_3\}\in E$, where $v_k\in V_k$, $k= 1,2,3$,  we
\begin{itemize}
\item
introduce 4 new vertices $e_i, i\in X$; we let $e_i\in V_i(H')$,
\item
introduce 2 new edges $E_1(e)= e\cup \{e_4\}$ and $E_2(e)= \{e_1, e_2, e_3,e_4\}$; let $w'(E_1(e))= w(e), w'(E_2(e))= 0$,
\item
introduce 36 new vertices:
for each $e_k$, $k= 1,2,3$, and $j \in X\setminus \{k\}$ we (1) introduce vertex $e^j_k\in V_j(H')$ and (2) for each $e^j_k$ and $l\in X\setminus \{j\}$ introduce
vertex $e^j_k(l)\in V_l(H')$,
\item
introduce 12 new edges of weight $0$: for each $e_k$, $k= 1,2,3$, let $F_k(e)= \{e_k\}\cup \{e^j_k; j\in X\setminus \{k\}\}$ and for  each $e^j_k$ let
$F^j_k(e)= \{e^j_k\}\cup \{e^j_k(l); l\in X\setminus \{j\}\}$.
\end{itemize}
New (14) edges introduced in this step form a 'tree' where each vertex has degree at most two; the vertices of degree 1 of this 'tree' are the three vertices of the original edge $e$ and the vertices   $\{e_k^j(l); k= 1,2,3, j\in X\setminus \{k\}, l\in X\setminus \{j\}\}$.

{\bf Step 2.} For each $k= 1,2,3$,  we 
\begin{itemize}
\item
introduce 3 new vertices $v_k^j; j\in X\setminus \{k\}$; we let $v_k^j\in V_j(H')$,
\item
For each $k= 1,2,3$ we introduce 3 disjoint edges $Q_k^1(e), Q_k^2(e), Q_k^3(e)$ covering the  vertices of
$\{v_k^j; j\in X\setminus \{k\}\cup  \{e_k^j(l); j\in X\setminus \{k\}, l\in X\setminus \{j\}\}$.
It is simple to see that such 4 edges always exist. We do not specify them to make the construction easier to follow.
\end{itemize}
This finishes the construction of $H', w'$.
Let $e\in E(H)$. We define two sets: 
$$
M_1(e)=\{E_1(e), F_1(e), F_2(e), F_3(e)\}\cup \{Q_k^m(e); k=1,2,3,  m=1,2,3 \},
$$
$$
M_2(e)= \{E_2(e)\}\cup \{F_k^j(e); k=1,2,3, j\in X\setminus \{k\} \}.
$$
If $Q$ is a perfect matching of $H$ then let $Q'= \Cup(M_1(e); e\in Q) \cup \Cup(M_2(e); e\notin Q)$. We observe that 
$Q$ is perfect matching of $H$ if and only if $Q'$ is a perfect matching of $H'$. 
Summarizing, 
$$
\Per(T(H,w))= \Per(T(H',w')).
$$ 

Next we show that $T(H',w',z)$ is Kasteleyn: by Theorem \ref{thm.k1} it suffices to observe that the
corresponding three bipartite graphs $G_1, G_2, G_3$ are planar. We first observe that in $H'$, the vertices of $V_i(H')$ which may have degree bigger than two belong to the set $B_i= V_i\cup \{v_k^i; k\neq i\}$. Hence each $G_i$
consists of internally disjoint paths $v_i\rightarrow v_i^4 (v_i\in V_i)$, $v_k^i\rightarrow v_k^4 (v_k\in V)$ and disjoint cycles
(of an even lengths) with no vertex from $B_i$. Hence each $G_i$ is planar.
This finishes the proof of Theorem \ref{thm.maindet}.

\end{proof}

 {\bf Proof of Theorem \ref{thm.main1}.}

It follows from Theorem \ref{thm.mainper}, Theorem \ref{thm.maindet} and Observation \ref{o.contr} that
$$
W_{\C, w}(z)= \det(I'+A'),
$$
where $A'$ is a signing of the adjacency matrix $A(D,t,z)$ and $I'$ is a signing of the identity matrix. Multiplying some 'hyper-rows' of $I'+A'$ by $(-1)$ we obtain the statement of Theorem \ref{thm.main1}.

\section{4D analogue of Bass' theorem}
\label{s.3bt}

We first reformulate the discrete Ihara-Selberg function in a way that is easier to generalise.

\subsection{Back to the discrete Ihara-Selberg function}
\label{sub.2}
\begin{definition}
\label{2-basic}
Let $D= (V,A)$ be a directed graph. Let $\A$ be a multiset of arcs, where each arc of $D$ appears infinitely many times.
A {\em vertex connector} in $D$ is a triple $(e_1, e_2, v)$ where $v$ is a vertex and $e_1, e_2$ are elements of $\A$ such that
$e_1$ terminates in $v$ and $e_2$ starts in $v$.

A {\em 2-circuit} of $D$ is a pair $(S, C)$ where $S\subset \A$ is a multi-set of arcs and $C$ is a set of vertex connectors so that
\begin{enumerate}
\item 
Each arc of $S$ is the entering arc of exactly one vertex connector of $C$ and also the leaving arc of exactly one vertex connector of $C$.
as t,
\item 
Each arc of each vertex connector of $C$ belongs to $S$,
\item
The digraph induced by the arcs of $S$ is weakly connected,
\item 
Each vertex of $V$ is in at most one vertex connector.
\end{enumerate}

A {\em 2-circulation} in $D$ is a pair $(S, C)$ satisfying (1), (2), (3) and 
\begin{enumerate}
\item[(5)]
It is not possible to write $S$ as disjoint union of $S_1, S_2$ and $C$ as disjoint union of $C_1, C_2$ such that both $(S_i,C_i), i= 1,2$, are 2-circulations.
\end{enumerate}
\end{definition}

 The following observation is straightforward.

\begin{observation}
\label{o.1}
Directed cycles of $D$ are exactly 2-circuits and closed walks of $D$ are exactly  2-circulations.  
\end{observation}

The following reformulation will be also helpful: given a digraph $D$ we define a new digraph so that its vertices are all the vertex connectors, and there is arc $a'= (c,c')$ for each arc $a\in \A$
 such that $a$ belongs to $c$ as leaving arc, and $a$ belongs to $c'$ as entering arc. We call the directed cycles of this new digraph {\em connector cycles} of $D$.

If $(S,C)$ is a 2-circulation then $m(S,C)$ will denote the number of the connector cycles of $(S,C)$. We note that this is only a formal definition which prepares ground for a 4-dimensional generalization since, due to the following straightforward observation, $m(S,C)= 1$ always.
 
\begin{observation}
\label{o.2}
Connector cycles of $D$ are exactly 2-circulations of $D$.

\end{observation}

Next we need to define {\em periodic 2-circulations}. 
 We say that a 2-circulation $(S,C)$ is {\em periodic} if there is $k>1$ and partitions
$S= S_1\cup \ldots \cup S_k$ and $C= C_1\cup \ldots, \cup C_k$ so that, after identifying the different copies of each arc of $D$,  the pairs $(S_i, C_i), i=1, \ldots, k$ are all equal to the same 2-circulation. A 2-circulation is {\em aperiodic} if it is not periodic.

We denote by $\Delta(D)$ the set of all aperiodic 2-circulations of $D$. 
Using these notions, the discrete Ihara-Selberg function can be expressed as follows:

\begin{observation}
\label{o.cll}
Let $D= (V,A)$ be a directed graph with no loops and let $x= (x_a)_{a\in A}$ be the vector of variables. Then
$$
IS_D(x)= 
\prod_{(S,C)\in \Delta(D)} (1+ (-1)^{m(S,C)} \prod_{a\in S}x_a).
$$ 

\end{observation}

Next we generalize the introduced concepts to 4 dimensions.

\subsection{Directed 4-hypergraphs}
\label{sub.3d}
Let $D=(V,A)$ be a directed 4-hypergraph and let $x= (x_a)_{a\in A}$ be a vector of variables associated with the edges of $D$. We will assume that the set of vertices of $D$ is {\em  linearly ordered}.

\begin{definition}
\label{def.or}
Let $a= (a_1, a_2, a_3, a_4)\in A$. We say that each $a_i$ is a vertex of $a$. We introduce new vertex $v(a)$ called {\em edge-vertex}, and four colored arcs: white $(v(a),a_1)$,  red $(v(a),a_2)$ and green $(v(a),a_3)$ start in $v(a)$ and blue $(a_4,v(a))$ enters $v(a)$. If $v$ is a vertex of a directed  hyperedge $a$ then the colored arc between $v(a)$ and $v$ will be denoted by $a(a,v)$.
\end{definition}

\begin{definition}
\label{def.3-basic}
Let $D= (V,A)$ be a directed  4-hyperraph and let $\A$ be the multiset of directed hyperedges of $D$ where each element of $D$ appears infinitely many times, with the same vertices. We usually denote the elements of $A$ by $a$ and the elements of $\A$ by $h$ or by $r$.
\begin{enumerate}
\item[(1)]
Let $S\subset \A$. Let $V(S)\subset V$ denote the set of vertices contained in at least one element of $S$.
We let $O(S)$ be the directed graph with the vertex-set $V(S)\cup \{v(h): h\in S\}$.  The arcs of $O(S)$ are exactly the colored arcs of the elements of $S$.
\item[(2)]
A {\em connector} in $D$ is a 5-tuple $(h_1, h_2, h_3, h_4, v)$ where $v\in V$ and  $h_1, h_2, h_3, h_4$ are directed
hyperedges of $\A$ incident with $v$ and such that $a(h_1,v)$ is white, $a(h_2,v)$ is red, $a(h_3,v)$ is green and $a(h_4,v)$ is blue.
Hence $a(h_4,v)$ leaves $v$ and the other three arcs enter $v$.
\end{enumerate}
\end{definition}

\begin{definition}
\label{def.4cc}
A {\em 4-circuit} of $D$ is a pair $(S, C)$ where $S\subset \A$ and $C$ is a {\em finite} set of connectors so that
\begin{enumerate}
\item 
For each $h=(v_1, v_2, v_3, v_4)\in S$ and $i\in \{1,2,3,4\}$ there is exactly one connector $c= (h_1, h_2, h_3, h_4, v)$ of $C$ such that $h= h_i$; moreover, if $h= h_i$ for $h_i$ of connector $c$ then $v= v_i$.
\item 
Each hyperedge of each connector of $C$ belongs to $S$,
\item
The digraph $O(S)$ induced by the arcs of the directed hyperedges of $S$ is weakly connected,
\item 
Each vertex of $V$ is in at most one connector.
\end{enumerate}

A {\em 4-circulation} in $D$ is a pair $(S, C)$  where $S\subset \A$ and $C$ is a {\em finite} set of connectors
 satisfying (1), (2), (3) and 
\begin{enumerate}
\item[(5)]
It is not possible to write $S$ as disjoint union of $S_1, S_2$ and $C$ as disjoint union of $C_1, C_2$ such that both $(S_i,C_i)$ are 4-circulation.
\item[(6)]
In addition we require that each 4-circulation $(S,C)$ satisfies the following {\em feasibility property}:  Let $v$ be a vertex of $V(S)$ and let $a$ be a blue arc of $S$ starting in $v$. Let $(V_a, O_a)$ be a subdigraph of $O(S)$
which is {\em minimal with respect to inclusion} with the properties: (1) $O_a$ contains arc $a$, (2) each edge-vertex $v(h)\in V_a$ is incident with one arc of $O_a$ of each of the four colors, (3) let $v'> v$ and a colored arc $a'$ of $O_a$ is incident with $v'$. Then all 4 colored arcs of the connector containing $a'$ belong to $O_a$, (4) let $v'< v$ and a {\em white, red or green} arc $a''$ of $O_a$ enters $v'$. Then the (three) white, red and green arcs of the connector containing $a''$ belong to $O_a$. Moreover the number of blue arcs of $O_a$ incident with $v'$ is equal to the number of white arcs of $O_a$ incident with $v'$. 

If each $v'< v$ is incident with at most one blue arc of $O_a$ then we require: 
\begin{itemize}
\item
$a$ is the only blue arc of $O_a$ incident with $v$.
\item
 Let $a_w$ be a white arc of $O_a$ incident with $v$, let $a_r$ be a red arc of $O_a$ incident with $v$ and let $a_g$ be a green arc of $O_a$ incident with $v$.
 Then we require that all three directed hyperdges whose edge-vertices are incident with $a_w, a_r, a_g$ belong to the same connector of $C$. We remark that this connector may be different from the connector of arc $a$.
\end{itemize}
\end{enumerate}
\end{definition}



Next, we define {\em periodic 4-circulations}. 
\begin{definition}
\label{def.perr}
We say that 4-circulation $(S,C)$ is {\em periodic} if there is $k>1$ and partitions
$S= S_1\cup \ldots S_k$, $C= C_1\cup \ldots, C_k$ so that, after identifying the different copies of each element of $\A$,  the pairs $(S_i, C_i), i=1, \ldots, k$ are all equal to the same 4-circulation. A 4-circulation is {\em aperiodic} if it is not periodic.
\end{definition}

We denote by $\Delta(D)$ the set of all aperiodic 4-circulations of $D$. 
Next we define {\em connector cycles}, as in 2D case above.

\begin{definition}
\label{def.3sign}
Let $D= (V,A)$ be a directed 4-hypergraph and let $\A$ be as above the multiset of directed hyperedges of $D$. We define a new digraph so that its vertex-set consists of all edge-vertices $v(h)$, $h\in \A$, and of all the connectors of $D$. There is an arc between connector $c$ and edge-vertex $v(h)$ if and only if  $a(v,h)$ is a colored arc of $c$; this arc will have the same orientation and color as $a(v,h)$.

We define
the {\em white}, ({\em red}, {\em green} respectively) {\em connector cycles} as directed cycles of this new digraph that contain alternately blue and white (red, green respectively) arcs.
 If $z$ is a connector cycle then we let 
$$
sign(z)=  (-1)^{|z|/2-1}.
$$
\end{definition}

Next, we turn our attention to 4D-matrices. Corollary \ref{p.3det1} below generalizes equation (\ref{eq.d}) from Introduction.

\begin{proposition}
\label{p.3det}
Let $A(D,x)$ be the adjacency 4-matrix of a directed 4-hypergraph $D= (V,A)$ with edge variables $x= (x_a)_{a\in A}$.
$$
\det(A(D,x))= \sum_{P= \{c_1, \ldots,  c_k\}} \prod_isign(c_i)\prod_{a\text{ hyperedge of } c_i}x_a,
$$
where $P$ ranges over all sets of vertex-disjoint 4-circuits $c_i$ which contain directed hyperedges of $A$ only and cover all vertices of $D$. Moreover $sign(c_i)= \prod_x sign(x)$, where
$x$ ranges over all connector cycles of $c_i$.
\end{proposition}
\begin{proof}
 We study the term of the defining expansion of  $\det(A(D,x))$ corresponding to a triplet of permutations $\alpha_1, \alpha_2$, $\alpha_3$. We find out that the set of the cycles of $\alpha_1$ corresponds to a collection of white connector cycles where each directed hyperedge belongs to $A$ only (we restrict the multiset $\A$ to $A$),  each edge-vertex is in at most one such cycle and each vertex of $V$ is in exactly one connector of these connector cycles. Moreover the sign of $\alpha$ is the product of the signs (defined in  Definition \ref{def.3sign}) of the connector cycles of the collection.

Let $Z$ be the set of collections $z$ of connector cycles satisfying:
\begin{enumerate}
\item  
If edge-vertex $v(a)$
belongs to a connector cycle of $z$ then (1) $a\in A$ and (2) for each color $\in \{$ white, red, blue $\}$, $v(a)$ belongs to unique connector cycle of $z$ of the color. 
\item
Each vertex of $V$ is in exactly one connector of a connector cycle of $z$. 
\item
Each connector of a connector cycle of $z$ belongs to exactly one connector cycle of each color. 
\end{enumerate}

If $z\in Z$ then let $R(z)$ be the set of the directed hyperedges of the elements of $z$ and let $sign(z)$ be the product of the signs of the connector cycles of the elements of $z$.
 It follows that
$$
\det(A(D,x))= \sum_{z\in Z}sign(z)\prod_{a\in R(z)}x_a.
$$

Clearly, the collection of connector cycles of a set of vertex-disjoint 4-circuits which cover all the vertices of $V$ belongs to $Z$.
In order to finish the proof, we argue that each $z\in Z$ can be partitioned into the sets of connector cycles of vertex-disjoint 4-circuits: indeed, the set of directed hyperedges of $z$ and the set of connectors of $z$ satisfy properties (1), (2), (4) of Definition \ref{def.4cc} of a 4-circuit.

\end{proof}

\begin{corollary}
\label{p.3det1}
Let $A(D,x)$ be the adjacency 4-matrix of a directed 4-hypergraph $D= (V,A)$ with edge-variables $x= (x_a)_{a\in A}$. 
$$
\det(I-A(D,x))= \sum_{Q= \{c_1, \ldots,  c_k\}}(-1)^{m(Q)}\prod_i\prod_{a\in c_i}x_a,
$$
where $Q$ ranges over all sets of vertex-disjoint 4-circuits $c_i$ and $m(Q)$ is the total number of connector cycles of $Q$.
\end{corollary}
\begin{proof}
We use Proposition \ref{p.3det}. It remains to prove that the signs are correct. Each $Q$ contributes $(-1)^{n(Q)}$, where
$n(Q)= \sum_{c\text{ connector cycle of $Q$ }} (|c|/2-1)+ r(Q)$, where $r(Q)$ is the number of directed hyperedges of $Q$.
We further notice that, since each hyperedge of $Q$ contributes 6 to the total length of the connector cycles, we have
$\sum_{c\text{ connector cycle of $Q$ }} |c|/2= 3r(Q)$.

\end{proof}

Now we introduce our 4-dimensional analogue of the discrete Ihara-Selberg function.

\begin{definition}
\label{def.4disf}
Let $D= (V,A)$ be a directed 4-hypergraph  and let $x= (x_a)_{a\in A}$ be the vector of edge-variables. The 4-dimensional discrete Ihara-Selberg function, $4IS_D(x)$, is the following formal product: 
$$
4IS_D(x)= 
\prod_{(S,C)\in \Delta(D)} (1+ (-1)^{m((S,C))}\prod_{h\in S}x_h,
$$
where $m((S,C))$ is the number of connector cycles of $(S,C)$.
\end{definition}

Our 4D analogue of Bass' theorem follows. Its proof is postponed to the last section.
 
\begin{theorem}
\label{thm.b}
Let $D= (V,A)$ be a directed 4-hypergraph and let $x= (x_a)_{a\in A}$ be the vector of edge-variables. Then $\det(I-A(D,x))= 4IS_D(x)$.
\end{theorem}

\subsection{Proof of Theorem \ref{thm.main}}
\label{sub.pm}
It follows from Theorem \ref{thm.main1} that 
$$
W_{\C, w}(z)= \pm\det(I+A'),
$$
where $A'$ is a {\em signing} of the adjacency matrix $A(D,t,z)$.
We can write  $W_{\C, w}(z)= \pm\det(I-A'')$, where $A''= -A'$ is (another) signing of $A(D,t,z)$. Due to
Theorem \ref{thm.b},
$$
W_{\C, w}(z)= \pm4IS_D(x)|_{x_a:= \pm z^{t(a)}}, 
$$
where the sign of $z^{t(a)}$ is given by the sign of $A''_a$.

\section{Proof of 4D analoque of Bass' theorem.}
\label{s.proof}

We will use ideas of \cite{S} and \cite{FL}.
Let $(S,C)$ be a 4-circulation of $D$ and let $c$ be a connector of $C$. Let $h\in \A$ be a hyperedge of $c$ and let $v$ be the vertex of $c$. By property (1) of Definition \ref{def.4cc} colored arc $a(h,v)$ uniquely determines connector $c$ of $(S,C)$ and we say that $a(h,v)$ {\em belongs to} connector $c$. We will also sometimes abuse notation and identify a colored arc of a connector cycle with the corresponding colored arc $a(h,v)$.
We also observe that connector cycles of $(S,C)$ of the same color are vertex disjoint.

We denote by $G$ the set of all finite subsets of $\Delta(D)$.
Let $t$ be a vertex of $D$. We denote by $\Delta_t$ the subset of $\Delta(D)$ of all aperiodic 4-circulations which have a connector containing vertex $t$.  Next, we introduce a procedure.

\medskip

 {\bf Procedure  P.} 

INPUT:
An aperiodic 4-circulation $(S,C)\in \Delta_t$ such that: if $v< t$ then $C$ has at most one connector containing $v$.  
A connector $c= (r_1, r_2, r_3, r_4, t)$ of $(S,C)$ and its blue arc $a$; we recall that $r_1, r_2, r_3, r_4$ are directed hyperedges of $D$ incident with $t$ and such that arc $a(r_1,t)$ is white, $a(r_2,t)$ is red, $a(r_3,t)$ is green and $a(r_4,t)= a$ is blue. Arc $a$ leaves $t$ and the other three colored arcs enter $t$.

We recall set $O_a$ from the feasibility property (Definition \ref{def.4cc}). Let $S_a$ be the set of directed hyperedges of $S$ whose edge-vertex is incident with at least one arc of $O_a$ and let $C_a$ be the set of the connectors which contain at least one arc of $O_a$. 

\medskip

{\bf Claim 1.} {\em $O_a$ contains 
 a white arc $a_w$, a red arc $a_r$ and a green arc $a_g$ which enter vertex $t$.}
\newline
 We observe that $O_a$ has all the arcs of the segment of the white (red, green respectively) connector cycle starting by $a$ and ending at $v$ (due to the defining propoerties of $O_a$). This proves Claim 1. 

\medskip

 By the feasibility property of 4-circulations, see Definition \ref{def.4cc},
we have that the three hyperedges of colored arcs $a_w, a_r, a_g$ of Claim 1  belong to the same connector of $(S, C)$. Let us denote this connector by $c'$.  

OUTPUT: Connector $c'$.
This finishes a description of procedure {\bf P}. 

\medskip

Next, we decompose $(S,C)\in \Delta_t$ into aperiodic 4-circulations $s_i= (S_i, C_i), i= 1, \ldots, k$ ($k$ depends on $(S,C)$) such that each $s_i$ contains unique connector whose vertex is $t$. The decomposition is constructed in two steps.
\medskip

{\bf STEP 1.} 
\newline
We start with a connector $c= c^1= (r^1_1, r^1_2, r^1_3, r^1_4, t)$. Let $a^1$ denote the blue edge of $c^1$. We let $i:= 1$.
\newline
Until STOP do:
We assume that pairs $(c^j,a_j), j\leq i$ are already constructed. 
We apply procedure ${\bf P}$ to $(c^i, a_i)$. If the output $c^{i'}$ satisfies $c^{i'}= c^1$ then we let $k= i$ and we return STOP. Otherwise
we let $c^{i'}= c^{i+1}$, we denote the blue arc of $c^{i+1}$ by $a_{i+1}$, and we update $i:= i+1$.
\newline
END of STEP 1.
We note that STEP 1 ends with $k$ at most the number of connectors of $(S,C)$ containing vertex $t$.

\medskip

{\bf Claim 2.} {\em $S_{a_i}\cap S_{a_j}= \emptyset$ for $i\neq j$ and $\cup_{i=1}^kS_{a_i}= S$.}
\newline
The second part holds due to property (5) of Definition \ref{def.4cc}. 
The first part follows if we show $O_{a_i}\cap O_{a_j}= \emptyset$: for a contradiction let $O= O_{a_i}\cap O_{a_j}\neq  \emptyset$. We observe that $a_i\notin O$ by the feasibility property for $a_j$. Then $O_{a_i}\setminus O$ satisfies conditions (1), (2), (3) of (6) of Definition \ref{def.4cc} which contradicts the minimality of $O_{a_i}$. 

\medskip 

{\bf STEP 2.}
\newline
We define pairs $s_i= (S_i,C_i)$ using pairs $(c^i, a_i)$ from STEP 1: 
\newline
For each $i\leq k$ let $S_i= S_{a_i}$.  For $i< k$ let $C_i= C_{a_i}\setminus\{c^i, c^{i+1}\}\cup
\{(r^{i+1}_1, r^{i+1}_2, r^{i+1}_3, r^i_4, t)\}$, and let $C_k= C_{a_k}\setminus\{c^k, c^1\}\cup \{(r^1_1, r^1_2, r^1_3, r^k_4, t)\}$. 
\newline
END of STEP 2.

\medskip

This finishes a description of the decomposition.  
The pairs $s_1, \ldots, s_k$ from the above construction will be called {\em stones} of $(S,C)$. We note that the decomposition into stones and the cyclic order of the stones induced by the construction is independent on the choice of connector $c^1$ we started with.  

 Each $s_i= (S_i, C_i), i= 1, \ldots, k$, is aperiodic since there is exactly one connector containing vertex $t$. We still need to show that each $s_i$ is a 4-circulation. We prove this in a more general setting in Proposition \ref{p.nn}. We also observe:

\begin{observation}
\label{o.mid}
$m((S,C))= [\sum_{i=1}^km((S_i, C_i))]- 3(k-1)$.
\end{observation}
\begin{proof}
In the construction of $(S_i, C_i)$, $1\leq i\leq k$, we replace one connector cycle of $(S,C)$ of each color (white, red, green) containing all blue arcs $a_1, \ldots, a_k$ by $k$ connector cycles of the same color, one from each $(S_i,C_i)$, $i= 1, \ldots, k$.

\end{proof}

Clearly, the circular order $(s_1, \ldots, s_k)$ of stones determines the aperiodic 4-circulation $(S,C)$ we started with by reversing STEP 2 of the above construction. Next we observe that each circular order of a subset of the stones defines, by reversing STEP 2 above, a 4-circulation. 

Let $(s'_1, \ldots, s'_m)$ be an arbitrary circular order of a subset of the stones. Let $s'_i= (S'_i, C'_i)$ and let 
$a'_i$ be the unique blue arc of $s'_i$ incident with $t$. Let us denote $c'_i= (q^i_1, q^i_2, q^i_3, q^i_4, t)$. Hence $a'_i= a(q^i_4, t)$. We let, for $i< m$,  $c''_i= (q^i_1, q^i_2, q^i_3, q^{i+1}_4, t)$ and $c''_m= (q^m_1, q^m_2, q^m_3, q^1_4, t)$. Finally let $S'= \Cup_{i\leq m}S'_i$ and $C'= (\Cup_{i\leq m}C'_i)\setminus \{c'_1, \ldots, c'_m\}\cup \{c''_1, \ldots, c''_m\}$.

\medskip

As an illustration of the above construction we observe: 

\medskip

If $m= 1$ then $(S', C')= (S_j,C_j)$ for some $j\leq k$. If  $(s'_1, \ldots, s'_m)= (s_1, \ldots, s_k)$ then $(S', C')= (S,C)$. In fact, the construction of $(S',C')$ is a generalization of the reverse operation to the decomposition above. 

\medskip

For each $i\leq m$ let $s''_i$ be obtained from $s'_i$ by identification of different copies of the same directed hyperedge $e$ of $A$. Let $\S$ denote the multiset consisting of $s''_1, \ldots, s''_m$.  

We will show:

\begin{proposition}
\label{p.nn}
$(S',C')$ is a 4-circulation. Moreover it is aperiodic (see Definition \ref{def.perr}) if and only if the circular sequence $(s''_1, \ldots, s''_m)$ is an aperiodic circular sequence of elements of $\S$.
\end{proposition}
\begin{proof}

We first observe that  $(S', C')$ satisfies (1), (2), (3) of Definition \ref{def.4cc}. In order to show property (5) we observe that the 4-circulation $Z$ contained in $(S', C')$ which has the hyperedge of arc $a'_1$ contains the hyperedge of each $a'_i, i\leq m$ (by the construction of $(S',C')$). Moreover, $Z$ also contains the hyperedge of each arc of $\cup S_{a'_i}= S'$ (by minimality of $O_{a'_i}$). Hence $Z= (S',C')$.

Hence it remains to observe that (6) of Definition \ref{def.4cc} holds. Let $a$ be a blue arc as in (6) of Definition \ref{def.4cc}. We observe directly from the definition that each set $O_a$ of $(S',C')$ is also set $O_a$ of $(S,C)$;  (6) for $(S',C')$ thus follows from (6) for $(S,C)$.

The second part of the proposition follows directly from the construction of $(S',C')$.

\end{proof}

Let $t$ be a vertex of $D$, let $T\subset V$ be the set of {\em proper predecessors} of $t$ in the fixed linear order of $V$. We denote by $U(T)$ the subset of $G$ consisting of all finite sets $Q$ of aperiodic 4-circulations of $D$ such that each vertex of $T$ appears at most once in a connector of an element of $Q$. In particular $U(\emptyset) = G$. We let
$U(T,t)= \{R\in U(T): \text{ no element of $R$ contains $t$}\}$ and for each $R\in U(T,t)$ we let
$U(T,R)= \{R': R'\cup R \in U(T)\text{ and each element of $R'$ contains $t$}\}$. We note that $\emptyset\in U(T,R)$.
A key ingredient of the proof is the following Lemma on coin arrangements of Shermann \cite{S} (see also \cite{L}).

\begin{lemma}(on coin arrangements)
Suppose we have a collection of $N$ objects of which $m_1$ are of one kind, $m_2$ are of second kind, ..., and $m_k$ are of k-th kind. Let $b_{N,i}= b(N, i; m_1, \ldots, m_k)$ be the number of exhausted unordered arrangements of these objects into $i$ disjoint, nonempty, circularly ordered sets such that no two circular orders are the same and none is periodic. If $N> 1$ then
$$
\sum_{i=1}^N(-1)^ib_{N,i}= 0.
$$ 

\end{lemma}

\begin{proposition}
\label{p.n2}
Let $t$ be a vertex of $D$, let $T\subset V$ be the set of {\em proper predecessors} of $t$ in the fixed linear order of $V$ and let $T'= T\cup\{t\}$.
$$
\sum_{Q= \{(S_1,C_1), \ldots, (S_k,C_k)\}\in U(T)} (-1)^{m(Q)} \prod_{i=1}^k\prod_{a\in S_i}x_a= \sum_{Q= \{(S_1,C_1), \ldots, (S_k,C_k)\}\in U(T')} (-1)^{m(Q)} \prod_{i=1}^k\prod_{a\in S_i}x_a
$$
\end{proposition}

\begin{proof} 
We can write 
$$
\sum_{Q= \{(S_1,C_1), \ldots, (S_k,C_k)\}\in U(T)} (-1)^{m(Q)} \prod_{i=1}^k\prod_{a\in S_i}x_a=
$$
$$
\sum_{R\in U(T,t)} [(-1)^{m(R)}\prod_{(S,C)\in R}\prod_{a\in S}x_a]
[\sum_{W\in U(T,R)} [(-1)^{m(W)}\prod_{(S,C)\in W}\prod_{a\in S}x_a.
$$
Now we rearrange the second summation by first breaking $W$ into stones and for such set of stones summing over all admissible arrangements, see Proposition \ref{p.nn}. The sign is calculated according to Observation \ref{o.mid}. We get that the expression above is equal to
$$
\sum_{R\in U(T,t)} [(-1)^{m(R)}\prod_{(S,C)\in R}\prod_{a\in S}x_a]
\sum_Z \beta(Z)\prod_{s\in Z}(-1)^{m(s)}\prod_{a\in s}x_a,
$$
where $Z$ ranges aver all finite multisets of stones (see the decomposition above) such that for each $v\in T$, $R\cup Z$ has at most one connector containing $v$. Moreover
$\beta(Z)= \sum_W(-1)^{3(|W|-1)}$ where the sum is over all exhausted unordered arrangements $W$ of the stones of $Z$ into disjoint, nonempty, circularly ordered sets such that no two circular orders are the same and none are periodic. It follows from the Lemma on coin arrangements that $\beta(Z)= 0$ whenever $|Z|>1$.  If $|Z|= 1$ then $\beta(Z)= 1$.
This proves the proposition.

\end{proof}

\begin{proof}(of 4D analogue of Bass' theorem, Theorem \ref{thm.b})
We have 
$$
\sum_{Q= \{(S_1,C_1), \ldots, (S_k,C_k)\}\in U(\emptyset)} (-1)^{m(Q)} \prod_{i=1}^k\prod_{a\in S_i}x_a= \sum_{Q= \{(S_1,C_1), \ldots, (S_k,C_k)\}\in U(V)} (-1)^{m(Q)} \prod_{i=1}^k\prod_{a\in S_i}x_a
$$
due to Proposition \ref{p.n2}. The left-hand-side is equal to $4IS_D(x)$ and the right-hand-side is equal to 
$\det(I-A(D,x))$ by Corollary \ref{p.3det1}.

\end{proof}

\bibliographystyle{amsplain} 

\end{document}